\documentclass[12pt]{amsart}

\title{Kissing numbers of regular graphs}
\author{Maxime Fortier Bourque and Bram Petri}
\date{\today}

\usepackage{amsmath,amsfonts,amssymb,amsthm,graphicx,overpic,hyperref}
\usepackage{float,enumitem}
\usepackage{color,xcolor}
\usepackage[colorinlistoftodos]{todonotes}
\usepackage{multirow}
\usepackage{pgfplots}
\usepackage{subcaption}
\pgfplotsset{compat=1.7}

\usepackage[margin=2.7cm]{geometry}  
\numberwithin{equation}{section}

\newtheorem{thm}{Theorem}[section]

\newtheorem{cor}[thm]{Corollary}
\newtheorem{lem}[thm]{Lemma}

\theoremstyle{definition}

\newtheorem{rem}[thm]{Remark}

\newtheorem{que}{Question}

\newcommand{\thmref}[1]{Theorem~\ref{#1}}

\newcommand{\secref}[1]{Section~\ref{#1}}

\newcommand{\lemref}[1]{Lemma~\ref{#1}}
\newcommand{\corref}[1]{Corollary~\ref{#1}}
\newcommand{\figref}[1]{Figure~\ref{#1}}

\newcommand{\eqnref}[1]{Equation~\eqref{#1}}

\newcommand{\nc}{\newcommand}
\nc{\dmo}{\DeclareMathOperator}
\usepackage{dsfont}

\nc{\abs}[1]{\left| #1 \right|}
\nc{\bigO}[1]{O\left(#1\right)}
\nc{\card}[1]{\left|#1\right|}
\nc{\ceil}[1]{\left\lceil #1 \right\rceil}
\nc{\CC}{\mathbb{C}}
\nc{\dilog}{\mathcal{L}}
\nc{\floor}[1]{\left\lfloor #1 \right\rfloor}
\nc{\ind}{\mathds{1}}
\nc{\ZZ}{\mathbb{Z}}
\nc{\len}[1]{\left| #1 \right|}
\nc{\littleo}[1]{o\left(#1\right)}
\dmo{\Mat}{Mat}
\nc{\NN}{\mathbb{N}}
\nc{\norm}[1]{\left|\left| #1 \right|\right|}
\nc{\QQ}{\mathbb{Q}}
\nc{\RR}{\mathbb{R}}
\nc{\st}[2]{\left\{\, #1 \,:\, #2\,\right\}}
\dmo{\supp}{supp}
\nc{\tr}[1]{\mathrm{tr}\left(#1\right)}
\nc{\what}{\widehat}
\dmo{\im}{Im}
\nc{\eps}{\varepsilon}
\dmo{\li}{li}
\dmo{\arccosh}{arccosh}

\dmo{\area}{area}
\dmo{\conv}{conv}
\dmo{\diam}{diam}
\dmo{\DD}{\mathbb{D}}
\dmo{\dist}{\mathrm{d}}
\nc{\HH}{\mathbb{H}}
\dmo{\Isom}{Isom}
\dmo{\MCG}{MCG}
\dmo{\MPL}{MPL}
\dmo{\Mod}{\mathcal{M}}
\dmo{\PL}{PL}
\nc{\Sphere}{\mathbb{S}}
\dmo{\sys}{sys}
\dmo{\kiss}{Kiss}
\dmo{\Teich}{\mathcal{T}}
\nc{\Torus}{\mathbb{T}}
\dmo{\vol}{vol}
\dmo{\WP}{WP}

\dmo{\convTV}{\;\stackrel{\mathrm{TV}}{\longrightarrow}\;}
\nc{\ExV}[2]{\mathbb{E}_{#1}\left[#2\right]}
\dmo{\EE}{\mathbb{E}}
\nc{\Pro}[2]{\mathbb{P}_{#1}\left[#2\right]}
\dmo{\PP}{\mathbb{P}}
\nc{\distTV}[2]{\mathrm{d}_{\rm TV}\left(#1,#2\right)}
\dmo{\UU}{\mathbb{U}}
\nc{\Var}[2]{\mathbb{V}\mathrm{ar}_{#1}\left[#2\right]}

\dmo{\alt}{\mathfrak{A}}
\dmo{\Aut}{Aut}
\dmo{\Fix}{Fix}
\dmo{\GL}{GL}
\dmo{\Hom}{Hom}
\dmo{\id}{Id}
\dmo{\PGL}{PGL}
\dmo{\PSL}{PSL}
\dmo{\PO}{PO}
\dmo{\Rep}{Rep}
\dmo{\SL}{SL}
\dmo{\SO}{SO}
\dmo{\sym}{\mathfrak{S}}
\dmo{\inv}{\mathcal{I}}
\dmo{\orb}{\mathcal{O}}
\dmo{\stab}{Stab}

\dmo{\calA}{\mathcal{A}}
\dmo{\calB}{\mathcal{B}}
\dmo{\calC}{\mathcal{C}}
\dmo{\calD}{\mathcal{D}}
\dmo{\calE}{\mathcal{E}}
\dmo{\calF}{\mathcal{F}}
\dmo{\calG}{\mathcal{G}}
\dmo{\calH}{\mathcal{H}}
\dmo{\calI}{\mathcal{I}}
\dmo{\calJ}{\mathcal{J}}
\dmo{\calK}{\mathcal{K}}
\dmo{\calL}{\mathcal{L}}
\dmo{\calM}{\mathcal{M}}
\dmo{\calN}{\mathcal{N}}
\dmo{\calO}{\mathcal{O}}
\dmo{\calP}{\mathcal{P}}
\dmo{\calQ}{\mathcal{Q}}
\dmo{\calR}{\mathcal{R}}
\dmo{\calS}{\mathcal{S}}
\dmo{\calT}{\mathcal{T}}
\dmo{\calU}{\mathcal{U}}
\dmo{\calV}{\mathcal{V}}
\dmo{\calW}{\mathcal{W}}
\dmo{\calX}{\mathcal{X}}
\dmo{\calY}{\mathcal{Y}}
\dmo{\calZ}{\mathcal{Z}}

\nc{\klav}{Klav\v{z}ar}

\nc{\bi}{\mathbf{i}}
\nc{\bj}{\mathbf{j}}
\nc{\bk}{\mathbf{k}}

\begin{document}

\begin{abstract}
We prove a sharp upper bound on the number of shortest cycles contained inside any connected graph in terms of its number of vertices, girth, and maximal degree. Equa\-lity holds only for Moore graphs, which gives a new characterization of these graphs. In the case of regular graphs, our result improves an inequality of Teo and Koh. We also show that a subsequence of the Ramanujan graphs of Lubotzky--Phillips--Sarnak have super-linear kissing numbers.
\end{abstract}

\maketitle

\section{Introduction}

We define the \emph{kissing number} $\kiss(G)$ of a graph $G$ to be the number of distinct shortest oriented cycles in $G$, whose common length is called the \emph{girth} of $G$. The kissing number can be similarly defined for any length metric space. Besides for graphs \cite{TeoKoh,Azakla}, this invariant has been studied for flat tori \cite{ConwaySloane,Vladut} and hyperbolic manifolds \cite{SchmutzKiss,SchmutzSuper,Par,FanPar,FBP}. The name comes from the fact that to any flat torus corresponds a sphere packing of its universal cover, and the kissing number represents the number of spheres tangent to any given sphere in this packing. For inhomogeneous spaces like finite graphs or hyperbolic manifolds, this picture is no longer accurate, but the name is kept by analogy.

The \emph{Moore bound} says that a $d$-regular graph of girth $g$ has at least
\[ 
1 + d \sum_{j=0}^{(g-3)/2} (d-1)^j \quad \quad \text{or} \quad  \quad 2 \sum_{j=0}^{(g-2)/2} (d-1)^j
\]
vertices depending on whether $g$ is odd or even, respectively. A graph achieving the Moore bound is called a \emph{Moore graph}. Apart from cycles, all Moore graphs have girth in the set $\{1,2,3,4,5,6,8,12\}$ \cite[Chapter 23]{BiggsBook}. In particular, there are only finitely many Moore graphs in each degree $d\geq 3$. All Moore graphs have been classified except for a hypothetical Moore graph of degree $57$ and girth $5$.

In this paper, we prove that graphs with large kissing numbers have large girth. The bound we prove is sharp, and becomes an equality precisely for Moore graphs.

\begin{thm}\label{thm:main}If $G$ is a connected graph of maximal degree $d$ and girth $g$ on $n$ vertices, then
\begin{equation} \label{eq:main}
\kiss(G) \leq \frac{n d (d-1)^{\lfloor g/2 \rfloor}}{g}
\end{equation}
with equality if and only if $G$ is a Moore graph.
\end{thm}

Equality can be interpreted to hold if $G$ is a point or an edge (in which case $g=\infty$). These degenerate cases can be included in the definition of Moore graphs for accuracy.

In \cite{TeoKoh}, Teo and Koh proved the similar inequality
\begin{equation} \label{eq:teokoh}
\kiss(G) \leq \begin{cases} 2n(m-n +1)/g & \text{if }g\text{ is odd} \\  2m(m-n +1)/g & \text{if }g\text{ is even} \end{cases} 
\end{equation}
for all $2$-connected graphs of girth $g\geq 3$ on $n$ vertices with $m$ edges. The bound is also attained for Moore graphs, but it is still unknown whether any irregular graph of even girth $g\geq 6$ and minimal degree at least $3$ can achieve equality. The paper \cite{TeoKoh} is closely related to earlier work of Homobono and Peyrat \cite{HP} on graphs such that every two edges are contained in a shortest cycle.

Apparently unaware of \cite{TeoKoh}, Azarija and \klav{} \cite{Azakla} showed that the number of convex oriented cycles in any graph of girth $g \geq 3$ with $n$ vertices and $m$ edges is at most $2n(m-n +1)/g$, with equality if and only if $G$ is a Moore graph of odd girth or an even cycle. Since shortest cycles of odd length are convex, their result implies the odd case of inequality \eqref{eq:teokoh}.

If we restrict to regular graphs, then inequality \eqref{eq:main} is stronger than \eqref{eq:teokoh} except for Moore graphs where they agree. It is also more general since it does not require $G$ to be simple nor $2$-connected. On the other hand, inequality \eqref{eq:teokoh} is better than \eqref{eq:main} for some irregular graphs such as complete bipartite graphs $K_{p,q}$ with $p\neq q$.

The proof of \thmref{thm:main} is both short and elementary. It essentially mimics Parlier's proof of an analogous inequality for hyperbolic surfaces \cite{Par} but in a simpler context.

Using the Moore bound, we can eliminate the girth from inequality \eqref{eq:main} and obtain a quantity that depends only on the degree and the number of vertices of the graph.
\begin{cor} \label{cor:main}
Let $d\geq 3$. If $G$ is a connected $d$-regular graph of girth $g\geq3$ on $n$ vertices, then
\begin{equation} \label{eq:log1}
\kiss(G) \leq  \frac{n(n(d - 2)+2)}{2\log_{d-1}\left(\frac{n(d-2) + 2}{d}\right)+1} \quad\text{if }g\text{ is odd}
\end{equation}
and
\begin{equation}  \label{eq:log2}
\kiss(G) \leq  \frac{nd(n(d -2)+2)}{4\log_{d-1}\left(\frac{n(d-2) + 2}{2}\right)} \quad\text{if }g\text{ is even,}
\end{equation}
with equality if and only if $G$ is a Moore graph.
\end{cor}

In particular, the kissing number of a regular graph is bounded by a sub-quadratic function of its number of edges (equal to $nd/2$). This should be compared with the fact that kissing numbers of closed hyperbolic manifolds grow at most sub-quadratically with their volume \cite{Par,FBP}. 

Note that there are infinite sequences of regular graphs achieving equality in \corref{cor:main}, namely, Moore graphs, but their degree is necessarily unbounded. So an inte\-resting question is: How close to the bound can large graphs get if their degree is fixed? For any $d\geq 3$, it is easy to construct a sequence of $d$-regular graphs whose kissing numbers grow linearly with the number of vertices. For example, Cayley graphs have this property since they are vertex-transitive. It is much less obvious that super-linear growth rate can be achieved. Indeed, it follows from \thmref{thm:main} that graphs with super-linear kissing numbers must have logarithmically large girth, and finding such graphs is a notoriously difficult problem \cite{ErdosSachs,Margulis,Imrich,Weiss,LPS,Margulis2,Morgenstern,BiggsCubic,Dahan}. In no way does having large girth imply having large kissing number in general, but it turns out that a subsequence of the current record holders for girth, the Ramanujan graphs of Lubotzky--Phillips--Sarnak \cite{LPS}, do have super-linear kissing numbers.

\begin{thm}\label{thm_kiss_ramanujan_intro}
For every prime number $p\equiv 1 \mod 4$ there is a subsequence $(X^{p,q_k})_k$ of the $(p+1)$-regular graphs $X^{p,q}$ of Lubotzky--Phillips--Sarnak such that 
\[ \lim_{k\to\infty} \frac{\log(\kiss(X^{p,q_k}))}{\log(n_k)} = \frac{4}{3}\]
where $n_k$ is the number of vertices of $X^{p,q_k}$.
\end{thm}

The definition of the graphs $X^{p,q}$ will be given in Section \ref{sec:ramanujan} together with the proof of the above theorem. In \secref{sec:data}, we gather some numerical data on all connected simple cubic graphs with at most $24$ vertices. We compare their kissing number with their girth, diameter, and number of automorphisms. The data suggests that graphs with large kissing number have large girth, small diameter, and many automorphisms even when they are not Moore. We conclude the paper with some open questions in \secref{sec:questions}.

\section{Proof of the kissing bound}

We start with a bit of terminology. The graphs we consider are multigraphs, where loops and multiple edges are allowed. A \emph{walk} in a graph is a finite sequence of oriented edges such that the initial vertex of each edge after the first is equal to the terminal vertex of the previous edge. A walk is \emph{closed} if it starts and ends at the same vertex. Closed walks are considered up to cyclic permutations of their sequence of edges. A \emph{cycle} is a closed walk which is embedded, i.e., in which each vertex is the initial vertex of at most one edge in the walk. A walk is \emph{geodesic} if it does not backtrack, i.e., if it never traces an edge to immediately retrace it in the opposite direction. A \emph{closed geodesic} is a closed walk which is locally geodesic. The \emph{girth} of a graph $G$ is equal to the smallest length that a closed geodesic in $G$ can have. A shortest closed geodesic is necessarily a cycle.

\begin{lem} \label{lem:fellow_travel}
Let $G$ be a graph of girth $g$. If two closed geodesics of length $g$ in $G$ share a subwalk of length $\lfloor g/2 \rfloor+1$, then they coincide.
\end{lem}
\begin{proof}
Suppose that $\alpha$ and $\beta$ are closed geodesics of length $g$  in $G$ that share a maximal geodesic subwalk $\omega$ of length at least $\lfloor g/2 \rfloor+1$ but less than $g$. By maximality of $\omega$, the concatenation of the walks $\alpha \setminus \omega$ and $\beta \setminus \omega$ is non-backtracking, hence is a closed geodesic. Furthermore, it has length at most $2(g - \lfloor g/2 \rfloor-1) < g$, which is a contradiction.
\end{proof}

\begin{rem}
For hyperbolic surfaces, the corresponding statement says that distinct shortest closed geodesics cannot intersect with too small an angle \cite[Lemma 2.4]{Par}.
\end{rem}

We are now ready to prove our main result.

\begin{proof}[Proof of \thmref{thm:main}]
By definition, $g \kiss(G)$ is equal to the number of shortest closed oriented geodesics in $G$ equipped with a starting point. From such a closed geodesic $\gamma$, we can obtain a geodesic walk $p(\gamma)$ by retaining its first $\lfloor g/2 \rfloor+1$ edges after the starting point. By \lemref{lem:fellow_travel}, the map $p$ is injective. The total number of geodesic walks of length $\lfloor g/2 \rfloor+1$ in $G$ is at most $nd (d-1)^{\lfloor g/2 \rfloor}$ since there are $n$ choices for the starting vertex, at most $d$ choices for the first edge, and at most $(d-1)$ choices for each of the subsequent edges. This proves that $g\kiss(G) \leq nd (d-1)^{\lfloor g/2 \rfloor}$.

If equality holds, then all of these choices are available, which means that $G$ is $d$-regular. The map $p$ must also be surjective, so that every geodesic walk of length $\lfloor g/2 \rfloor+1$ in $G$ is contained in a closed geodesic of length $g$.

To prove that $G$ is a Moore graph, consider an open ball $B$ of radius $g/2$ centered at any vertex in $G$ if $g$ is odd, or at any midpoint of an edge if $g$ is even. The ball is a tree, for otherwise we could find a closed geodesic of length less than $g$. We also claim that the closure of $B$ is equal to $G$. Otherwise, there is some open half-edge $e$ in the complement of $B$ whose closure intersects $\partial B$ (this uses the hypothesis that $G$ is connected). By the previous paragraph, the geodesic walk of length $\lfloor g/2 \rfloor+1$ that goes from the center of $B$ to its boundary and continues through $e$ can be extended to a closed geodesic of length $g$. This implies that all the points in $e$ are within distance $g/2$ from the center of $B$, i.e., that $e \subset B$. This is a contradiction, from which we conclude that $\overline{B}=G$, or equivalently that $B$ contains all the vertices in $G$ contains all the vertices in $G$. This means that $G$ is a Moore graph (recall that the Moore bound is obtained by counting the vertices in $B$).

Conversely, suppose that $G$ is a Moore graph. Then any geodesic walk of length $\lfloor g/2 \rfloor+1$ in $G$ exits and re-enters the closed ball $\overline B$ of radius $g/2$ centered at its starting point since $\overline B= G$. It can therefore be extended to a closed walk $\gamma$ of length $g$ by continuing towards the center of $\overline B$. Since $\gamma$ has at most one backtrack (at the center of $\overline B$), it cannot be contracted to a point. It follows that $\gamma$ is a closed geodesic, for otherwise the closed geodesic in its homotopy class has length strictly less than $g$. We have shown that every geodesic walk of length $\lfloor g/2 \rfloor+1$ in $G$ is contained in a closed geodesic of length $g$. This proves that the map $p$ from the first paragraph is a surjection, hence that $g\kiss(G) = nd (d-1)^{\lfloor g/2 \rfloor}$.
\end{proof}

Let us compare our bound with the one from \cite{TeoKoh} for regular graphs. Let $G$ be a connected $d$-regular graph of girth $g$ with $n$ vertices and $m=nd/2$ edges. If $d=2$, then $G$ is a cycle (hence a Moore graph) and the two bounds coincide. Thus, we may assume that $d > 2$. If $g$ is odd, then the Moore bound gives
\[
n \geq 1 + d \sum_{j=0}^{(g-3)/2} (d-1)^j = 1 + d \,\frac{(d-1)^{(g-1)/2} - 1}{d-2}  \quad \text{or} \quad n(d-2)+2 \geq d  (d-1)^{\lfloor g/2 \rfloor}
\]
and hence
\[
\frac{2n(m-n +1)}{g} = \frac{n(n(d-2)+2)}{g}  \geq \frac{nd(d-1)^{\lfloor g/2 \rfloor}}{g}
\]
with equality if and only if $G$ is a Moore graph. Similarly, if $g$ is even then the Moore bound gives
\[
n \geq 2 \sum_{j=0}^{(g-2)/2} (d-1)^j = 2 \, \frac{(d-1)^{g/2} - 1}{d-2} \quad \text{or} \quad n(d-2)+2 \geq 2 (d-1)^{\lfloor g/2 \rfloor}
\]
so that
\[
\frac{2m(m-n +1)}{g} = \frac{nd(n(d-2)+2)}{2g} \geq \frac{nd(d-1)^{\lfloor g/2 \rfloor}}{g}
\]
with equality if and only if $G$ is a Moore graph.

We conclude that for regular graphs that are not Moore graphs, inequality \eqref{eq:main} is better than \eqref{eq:teokoh}, and the discrepancy is proportional to the excess in the Moore bound. 

Next, we prove \corref{cor:main}, which is essentially a calculation.

\begin{proof}[Proof of \corref{cor:main}]
Simple algebra shows that the Moore bound is equivalent to
\begin{equation}  \label{eq:moore1}
g\leq 2 \log_{d-1} \left(\frac{n(d-2) + 2}{d}\right) + 1 
\end{equation}
if $g$ is odd and
\begin{equation}   \label{eq:moore2}
g\leq 2 \log_{d-1} \left(\frac{n(d-2) + 2}{2}\right)
\end{equation}
if $g$ is even. Moreover, the functions $x \mapsto (d-1)^{x}/(2x+1)$ and $y \mapsto (d-1)^{y}/(2y)$ are strictly increasing for $x\geq 1$ and $y\geq 2$. This means that we can replace the girth $g$ in inequality \eqref{eq:main} with the right-hand side of \eqref{eq:moore1} or \eqref{eq:moore2} (depending on parity) and still obtain a valid inequality. If either of the resulting weaker inequalities \eqref{eq:log1} or \eqref{eq:log2} becomes an equality, then so does the corresponding \eqref{eq:moore1} or \eqref{eq:moore2}, and $G$ is a Moore graph. The converse follows from \thmref{thm:main} and algebra.
\end{proof}

\section{Regular graphs with super-linear kissing numbers} \label{sec:ramanujan}

In this section we prove that there are infinite sequences $(G_k)_k$ of regular graphs $G_k$ of fixed degree on $n_k$ vertices whose kissing numbers grow faster than $n_k^{4/3-\eps}$ for any $\eps>0$. The graphs we use form a subsequence of the Ramanujan graphs from \cite{LPS}. 

For each pair of unequal primes $p\equiv q \equiv 1 \mod 4$, Lubotzky, Phillips, and Sarnak construct a $(p+1)$-regular graph denoted $X^{p,q}$. It turns out that if the Legendre symbol $\left(\frac{p}{q}\right)$ is equal to $-1$, which we will assume throughout, then $X^{p,q}$ is a bipartite Cayley graph for $\PGL(2,\ZZ/q\ZZ)$ \cite[Proposition 3.3]{LPS}. In particular, the number of vertices $\card{X^{p,q}}$ of $X^{p,q}$ is equal to $\card{\PGL(2,\ZZ/q\ZZ)} = q(q^2-1).$

These graphs have many extremal properties. To name a few: 
\begin{itemize}
\item Their spectral gap is essentially as large as possible  \cite{Alon}\cite[Theorem 4.1]{LPS} (a $d$-regular graph is \emph{Ramanujan} if its spectral gap is at least $2\sqrt{d-1}$). 
\item Their girth is very large. Indeed, their asymptotic girth ratio
\[ \limsup_{q\to \infty} \frac{g(X^{p,q})}{\log_p(\card{X^{p,q}})} = \frac{4}{3} \]
 \cite[Theorem 3.4]{LPS} is the largest known so far.

\item Their diameter is very small \cite[Theorem 5.1]{LPS}, although not as small as possible \cite{Sardari} (asymptotically, random regular graphs achieve the theoretical lower bound on diameter coming from \eqref{eq:diam}  \cite{BF}).
\end{itemize}

Regarding the kissing numbers of the graphs $X^{p,q}$, we will show the following upper and lower bounds.

\pagebreak

\begin{thm}\label{thm_kiss_ramanujan} Let $p\equiv 1 \mod 4$ be a prime number. Then 
\begin{itemize}
\item[(a)] for every $\eps>0$ we have
\[ \kiss(X^{p,q}) \leq \card{X^{p,q}}^{4/3+\eps} \]
for all large enough prime numbers $q \equiv 1 \mod 4$ such that $\left(\frac{p}{q}\right)=-1$ and
\item[(b)] there exists an infinite sequence of prime numbers $(q_k)_k$ satisfying $q_k\equiv 1 \mod 4$ and $\left(\frac{p}{q_k}\right)=-1$ for all $k\in \NN$ such that for every $\eps>0$ we have 
\[ \kiss(X^{p,q_k}) \geq \card{X^{p,q_k}}^{4/3-\eps}\]
provided that $k \in \NN$ is large enough.
\end{itemize}
\end{thm}

Theorem \ref{thm_kiss_ramanujan_intro} stated in the introduction is an immediate consequence of Theo\-rem \ref{thm_kiss_ramanujan}.

\begin{rem}
Since the girth of $X^{p,q}$ is known when $\left(\frac{p}{q}\right)=-1$ (see \eqnref{eq:girth} below), we can compare the above result with what comes out of Theorem \ref{thm:main}. After substituting the girth and simplifying, the latter predicts that
\[ \kiss(X^{p,q}) \leq \frac{q(q^2-1) p q^2}{4 \log_p(q)} \sim \frac{3p}{4} \frac{\card{X^{p,q}}^{5/3}}{\log(\card{X^{p,q}})} \text{ as } q\to \infty \]
which is considerably larger than the actual asymptotic growth given in \thmref{thm_kiss_ramanujan_intro}.
\end{rem}

We will now describe the graphs $X^{p,q}$ and some of their properties in more details. After that, we will recall some classical facts about the number of integer solutions to certain quadratic equations and at the end of this section, we will combine these with some elementary modular arithmetic to obtain a proof of Theorem \ref{thm_kiss_ramanujan}.

\subsection{The graphs}

We very briefly recall the construction from \cite{LPS} (see also \cite{Lubotzky,Sarnak,DSV}). The goal is to construct a graph $X^{p,q}$ for all unequal primes $p,q$ congruent to $1 \mod 4$ such that $\left(\frac{p}{q}\right)=-1$. These graphs can be defined more generally for primes $q$ with $\left(\frac{p}{q}\right)=1$ \cite{LPS} and for non-primes $q$ \cite{Lubotzky}, but the case we consider makes it easier to control the girth and the number of vertices.

Let 
\[ \HH(\ZZ) = \st{x_0+x_1\bi+x_2\bj+x_3\bk}{x_0,x_1,x_2,x_3 \in \ZZ}, \quad \bi^2=\bj^2=\bk^2 = \bi\bj\bk = -1\]
denote the integral Hamiltonian quaternions. The norm of any $x\in \HH(\ZZ)$ is defined as
\[ N(x) = x_0^2+x_1^2+x_2^2+x_3^2.\]

Now consider the set
\[ \Lambda(2) = \st{x\in\HH(\ZZ)}{ x\equiv 1 \mod 2 \text{ and } N(x) = p^k \text{ for some } k\in \NN}\big/\sim\]
where 
\[ x\sim y \quad \Longleftrightarrow \quad x = \pm p^k y \text{ for some } k\in\ZZ \]
and $\NN$ is the set of non-negative integers. Multiplication in the quaternions induces a group structure on $\Lambda(2)$. Lubotzky, Phillips and Sarnak prove that $\Lambda(2)$ is a free group with free symmetric generating set 
\[ S = \st{x=x_0+x_1\bi+x_2\bj+x_3\bk \in \HH(\ZZ)}{ N(x)=p,\; x_0>0 \text{ and } x_0\equiv 1 \mod 2 }/\sim.\]
It follows from Jacobi's theorem on the number of representations of an integer as the sum of four squares that $\card{S} = p+1$.

We write 
\[ \Lambda(2q) = \st{[x_0+x_1\bi+x_2\bj+x_3\bk]\in \Lambda(2)}{ 2q | x_j,\;j=1,2,3}.\]
This is a normal subgroup of $\Lambda(2)$.

The graphs of Lubotzky--Phillips--Sarnak are defined to be the Cayley graphs
\[X^{p,q} = \mathrm{Cay}(\Lambda(2)/\Lambda(2q), S). \]
We will also write
\[X^p = \mathrm{Cay}(\Lambda(2), S) \]
for the infinite $(p+1)$-regular tree.

\subsection{Closed geodesics}

Since $X^{p,q}$ is a Cayley graph, it is vertex-transitive. So, in order to understand the closed geodesics in this graph, it suffices to study those based at our favorite vertex, the identity element in $X^{p,q}$.

A geodesic walk starting from $[1]$ in $X^p$ projects to a closed geodesic in $X^{p,q}$ if and only if its endpoint lies in $\Lambda(2q)$. Moreover, the length of a geodesic walk can be read out from the norm of an appropriate representative of its endpoint. Indeed, the norm on $\HH(\ZZ)$ is multiplicative and the generators in $S$ all have representatives of norm $p$. Therefore, if the geodesic walk has length $k$ and endpoint $\xi \in \Lambda(2)$ then
\[ N(x_0+x_1\bi+x_2\bj+x_3\bk) = p^k\]
where $x_0+x_1\bi+x_2\bj+x_3\bk \in \HH(\ZZ)$ is the unique representative of $\xi$ such that
\[ x_0 >0 \quad \text{and} \quad p \nmid \gcd(x_0,x_1,x_2,x_3).\]

On the other hand, given an integer solution to the equation
\begin{equation}\label{eq_quadratic_form}
 y_0^2 + (2q)^2 ( y_1^2 + y_2^2 + y_ 3^2) = p^k, \quad y_0>0,  \quad p \nmid \gcd(y_0,y_1,y_2,y_3)
\end{equation}
we obtain a closed geodesic of length $k$ based at the identity in $X^{p,q}$ by projecting the unique geodesic walk from $[1]$ to $[y_0+2qy_0\bi+2qy_0\bj+2qy_0\bk]$  in $X^p$.

Thus, the number of closed geodesics of length $k$ through the identity in $X^{p,q}$ is the number of integral solutions to \eqnref{eq_quadratic_form}. Note that all closed geodesics in $X^{p,q}$ have even length since this graph is bipartite.

\subsection{Quadratic forms}

In order to control the number of shortest closed geodesics in the graphs $X^{p,q}$, we need to bound the number of solutions to \eqnref{eq_quadratic_form}. For this, we start by recalling some facts about the number of representations of an integer as the sum of three squares. 

First of all, for any positive integer $n$, we set
\[ r_3(n) = \card{\st{(a,b,c)\in\ZZ^3}{a^2+b^2+c^2 =n}}.\]
Legendre's theorem states that $r_3(n)>0$ if and only if $n\notin 4^\NN (8\NN+7)$, where $\NN$ is again the set of non-negative integers. We will require the following more precise estimates.

\begin{lem} \label{lem:BatemanSiegel}
For every $\eps>0$, there exists a constant $C_\eps>0$  such that for every positive integer $n$ we have
\begin{equation}\label{eq_rep_three_squares_upper}
r_3(n) \leq C_\eps n^{1/2+\eps}.
\end{equation}
If furthermore $r_3(n) \neq 0$ and $16 \nmid n$, then
\begin{equation}\label{eq_rep_three_squares_lower}
r_3(n) \geq \frac{1}{C_\eps} n^{1/2-\eps}.
\end{equation}
\end{lem}

\begin{rem}
The number $16$ above could be replaced by any power of $4$, but this particular statement will suffice for our needs.
\end{rem}

\begin{proof}
Bateman \cite[Theorem B]{Bateman} (see also \cite[\S 4.10]{Grosswald} for an exposition) proved the exact formula
\[ r_3(n) = \frac{16}{\pi}\, \sqrt{n}\; L\left(1,\left( \frac{-4n}{\cdot}\right)\right) P(n) \, Q(n) \]
where the functions $L$, $P$ and $Q$ are defined as follows. First,
\[ L\left(s,\left( \frac{-4n}{\cdot}\right)\right) = \sum_{m=1}^\infty \left( \frac{-4n}{m} \right) \frac{1}{m^s}\]
for all $s\in \CC$ for which the right-hand side converges. This is the $L$-function associated to the Kronecker symbol $\left( \frac{-4n}{\cdot}\right)$, a generalization of the Legendre symbol. Second,
\[P(n) = \prod_{\substack{p^2\mid n \\ p \text{ an odd prime}}} \left(1 + \sum_{j=1}^{b_{p,n}-1}p^{-j} + p^{-b_{p,n}}\cdot\left(1- \left(\frac{-p^{-2b_{p,n}}n}{p} \right) \frac{1}{p} \right)^{-1} \right)
\]
where $b_{p,n}=\max\st{b\in\NN}{ p^{2b} \mid n }$. 
Finally, if we write $n = 4^a n_1$ such that $4\nmid n_1$ then
\[ Q(n) = \left\{
\begin{array}{ll}
0 & \text{if } n_1 \equiv 7 \mod 8, \\
2^{-a} & \text{if } n_1 \equiv 3 \mod 8,\\
3 \cdot 2^{-a-1} & \text{if } n_1 \equiv 1,2,5,6 \mod 8.
\end{array}
 \right.\]

Note that $P(n) \geq 1$ and $Q(n)\leq 3/2$ for all $n\geq 1$. We also have $P(n) \leq C (\log n)^2$ for some constant $C>0$. Indeed,
\begin{multline*}
P(n) \leq \prod_{\substack{p^2\mid n \\ p \text{ an odd prime}}} \left( \sum_{j=0}^{b_{p,n}-1}p^{-j} + p^{-b_{p,n}}\cdot \frac{1}{1-\frac{1}{p}} \right)
\leq   \prod_{\substack{p^2\mid n \\ p \text{ an odd prime}}} \left(\sum_{j=0}^{b_{p,n}} p^{-j} \right) \prod_{\substack{p \leq  n \\ p \text{ a prime}}} \frac{1}{1-\frac{1}{p}}
\end{multline*}
and each product on the right is essentially bounded by $\log n$ since
\[ \prod_{\substack{p^2\mid n \\ p \text{ an odd prime}}} \left(\sum_{j=0}^{b_{p,n}} p^{-j} \right) \leq \sum_{k=1}^{\floor{\sqrt{n}}}\frac{1}{k} \leq \frac12\log n +1 \quad \text{and} \quad \prod_{\substack{p \leq  n \\ p \text{ a prime}}} \frac{1}{1-\frac{1}{p}} 
\stackrel{n\to\infty}{\sim} e^\gamma \log n, \]
where the first inequality is obtained by expanding the product and unique factorization, and the asymptotic is Mertens's third theorem ($\gamma$ denotes the Euler-Mascheroni constant).  If we further know that $16 \nmid n$, then either $r_3(n) = 0$ or $Q(n)\geq 1/2$. 

Siegel \cite[p. 406-409]{Siegel} showed that for every $\eps > 0$, there exists a constant $B_\eps>0$ such that 
\[ \frac{1}{B_\eps\,n^\eps} \leq L\left(1,\left( \frac{-4n}{\cdot}\right) \right) \leq B_\eps\,n^\eps \]
for all $n\geq 1$. Combining this with the above upper and lower bounds on $P$ and $Q$ yields the desired results (the logarithmic factors in the upper bound can be absorbed in $n^\eps$).
\end{proof}

\subsection{The girth}

A key ingredient needed in the proof of \thmref{thm_kiss_ramanujan} is the girth of $X^{p,q}$, which was determined in \cite{BiggsBoshier}. In that paper, Biggs and Boshier proved that for every class $\xi \in \Lambda(2q)$ at an even distance $k$ from the identity in $X^p$, there exists a representative $x=x_0 + x_1\bi+x_2\bj + x_3 \bk$ such that
\[ x_0 = |p^{k/2}-mq^2| \]
for some positive even number $m$ \cite[Lemma 1]{BiggsBoshier}. Writing $x=x_0 + 2q y_1 \bi+2q y_2\bj + 2q y_3 \bk$ as before, \eqnref{eq_quadratic_form} turns into
\begin{equation} \label{eq:special_rep}
 2mp^{k/2} - m^2 q^2 = 4 (y_1^2 + y_2^2 + y_3^2).
\end{equation}
By Legendre's three-square theorem, this equation has integer solutions if and only if 
\[2mp^{k/2} - m^2 q^2 \geq 0 \quad \text{and} \quad 2mp^{k/2} - m^2 q^2 \notin 4^{\NN+1} (8\NN+7). \]

On page 193 of \cite{BiggsBoshier}, Biggs and Boshier show that there exists a non-trivial element $\xi \in \Lambda(2q)$ closest to the identity in $X^p$ which has a representative as above with $m=2$ or $m=4$. Moreover, the girth of $X^{p,q}$ is equal to the smallest even number $k$ such that $p^{k/2} - q^2 \geq 0$ if \eqnref{eq:special_rep} a solution with $m=2$, and to the smallest even number $k$ such that $2p^{k/2} - 4q^2 \geq 0$ otherwise. Concretely, the girth is given by
\begin{equation} \label{eq:girth}
g = g(X^{p,q}) = \left\{
\begin{array}{ll}
2 \ceil{\log_p(q^2)} & \text{if } p^{\ceil{\log_p(q^2)}}-q^2 \notin 4^\NN (8\NN+7) \\[2mm]
2 \ceil{\log_p(2q^2)} & \text{otherwise}.
\end{array}
\right.
\end{equation}
Fix $m=2$ in the first case and $m=4$ in the second case. If we assume that $q>p$, then $p$ does not divide $x_0=|p^{g/2} - mq^2|$ so that all the integer solutions $(y_1,y_2,y_3)$ to
\eqnref{eq:special_rep} with $k=g(X^{p,q})$ correspond to distinct non-trivial elements \[[x_0 + 2q y_1 \bi+2q y_2\bj + 2q y_3 \bk]\in \Lambda(2q)\] closest to the identity in $X^p$. That is to say, the number of shortest closed geodesics through the identity in $X^{p,q}$ is at least
\begin{equation}\label{eq_loops_id}
\left\{ \begin{array}{ll}
r_3(p^{g/2} - q^2) &  \text{if } p^{g/2}-q^2 \notin 4^\NN (8\NN+7) \\
r_3(2p^{g/2} - 4q^2) & \text{otherwise}.
\end{array} \right.
\end{equation}
where $g$ is the girth of $X^{p,q}$.

\subsection{The upper bound}  

The upper bound on the kissing number of the graphs $X^{p,q}$ follows from the discussion in the last two subsections.

\begin{proof}[Proof of Theorem \ref{thm_kiss_ramanujan} (a)]
Assume that $q>p$ and let $g$ be the girth of $X^{p,q}$, given by \eqnref{eq:girth}. We will again use the fact that any class $\xi \in \Lambda(2q)$ corresponding to a shortest closed geodesic through the identity in $X^{p,q}$ can be represented by an integral quaternion $x=x_0+x_1\bi+x_2\bj+x_3\bk$ with $x_0 = |p^{g/2}-mq^2|$ for some even number $m$.
First, we want to bound the number of possibilities for $m$. Since 
\[ g \leq 2(\log_p(2q^2)+1) \quad \text{and} \quad 2mp^{g/2}-m^2q^2 \geq 0\]
(see the previous subsection), this implies that
\[ (4mp -m^2)q^2\geq 0\]
and hence that $m\leq 4p$. So, we get that the number of shortest closed geodesics through the identity in $X^{p,q}$ is at most
\[ 4p\cdot \max_{\substack{2\leq m \leq 4p \\ \text{even}}} \left\{ r_3\left(2mp^{g/2} - m^2 q^2 \right) \right\} \leq 4p \cdot \max\st{r_3(n)}{ 0\leq n \leq \frac{7}{4}p^2 q^2},\]
where we used that $4mp-m^2 \leq 7p^2/4$. Now we use the upper bound \eqref{eq_rep_three_squares_upper} on the representation number and get that for every $\eps>0$ the number of shortest closed geodesics through the identity in $X^{p,q}$ is at most
\[ q^{1+\eps} \]
for all $q$ large enough. Multiplying by the total number of vertices (cubic in $q$), we obtain the upper bound we were after.

\end{proof}

\subsection{The lower bound}

Our goal is to find a sequence of primes $(q_k)_k$ that are equal to $1 \mod 4$ and render the expressions in \eqref{eq_loops_id} large. In view of \lemref{lem:BatemanSiegel}, it suffices to make sure that the argument of $r_3$ in \eqref{eq_loops_id} is not in the set  $4^\NN (8\NN+7)$ nor is divisible by sixteen. The following lemma will do the trick.

\begin{lem} \label{lem:modular}
For every large enough number $k$, there exists a prime $q_k \equiv 1 \mod 4$ such that $\left(\frac{p}{q_k} \right) = -1$, $q_k^2 \not\equiv p  \mod 16$ and $p^k < q_k < 2p^k$.
\end{lem}

\begin{proof}
We first claim that there exists a number $m\in \NN$ such that 
\[ \left(\frac{m}{p} \right)  = -1, \quad m^2 \not\equiv p \mod 16 \quad \text{and} \quad m\equiv 1 \mod 4.\]
This can be argued as follows. As the group of invertible elements in $\ZZ/p\ZZ$ is cyclic, half of its elements are squares and the other half are not, so we can fix a representative $L\in\NN$ of a non-square mod $p$. Since $p$ is invertible mod $16$, for every residue class $c \in \ZZ/16\ZZ$ there exists a number $a\in\NN$ such that
$L+ap \in c$. Our goal is to pick $c \in \ZZ/16\ZZ$ in such a way that the number $m=L+ap$ constructed in the previous sentence has the desired properties. For this, we need to choose $c=[j]$ such that $j\equiv 1\mod 4$ but $j^2 \not\equiv p \mod 16$. We can always find such a residue class:
\begin{itemize}
\item if $p \equiv 1 \mod 16$ then both $c=[5]$ and $c=[13]$ work,
\item if $p \equiv 9 \mod 16$ then we can pick $c=[1]$ or $c=[9]$, and
\item if $p\equiv 5 \text{ or }13 \mod 16$ then any of $[1]$, $[5]$, $[9]$ or $[13]$ is fine.
\end{itemize}
So we fix $c$ according to these rules, determine $a$ as above, and set $m=L+ap$. This number is congruent to $1$ mod $4$, its square is different from $p$ mod $16$, and it is not a quadratic residue mod $p$.  Note that these conditions are also satisfied for all members of the arithmetic progression $(m+ r \cdot 16 p)_{r\in \NN}$. Moreover, by the law of quadratic reciprocity, any prime number $q$ in this progression satisfies $\left(\frac{p}{q} \right)  = -1$ as well. Since $m$ and $16p$ are coprime, there are infinitely many such primes by Dirichlet's theorem. The Siegel-Walfisz theorem \cite{Walfisz} further states that the number of primes less than or equal to $x$ in this arithmetic progression is
\[ \frac{\mathrm{Li}(x)}{\varphi(16p)}  + \bigO{x \exp\left(- C (\log x )^{1/2}\right)},
\] 
where $\mathrm{Li}$ denotes the logarithmic integral, $\varphi$ the Euler totient function and $C>0$ is some constant depending on $p$ only. This implies the existence of a prime $q_k$  in the progression $(m+ r \cdot 16 p)_{r\in \NN}$ which is contained in the required interval if $k$ is large enough.
\end{proof}

The conditions on the sequence $(q_k)_k$ above also imply the following.

\begin{lem} \label{lem:sixteen}
Let $p\equiv 1 \mod 4$ be a prime number, and let $k$ and $q_k$ be as in \lemref{lem:modular} with $k$ even. Then $16 \nmid  p^{\ceil{\log_p(q_k^2)}} - q_k^2$ and $16 \nmid  2p^{\ceil{\log_p(2q_k^2)}} - 4q_k^2$.
\end{lem}
\begin{proof}
Let us analyze the two expressions above separately. We have
\[ p^{\ceil{\log_p(q_k^2)}} = p^{2k+1} \equiv p \mod 16,\]
where the first equality comes from the estimate on $q_k$ and the second from the fact that 
\[[1]^4=[5]^4=[9]^4=[13]^4=[1] \]
in $\ZZ/16\ZZ$. This, by the assumption that $q_k^2 \not \equiv p \mod 16$, means that 
\[16 \nmid  p^{\ceil{\log_p(q_k^2)}} - q_k^2. \]
For the second expression, we use the hypothesis that $p\equiv q_k \equiv 1 \mod 4$ to get
\[ 2p^{\ceil{\log_p(2q_k^2)}} - 4q_k^2  \equiv 2 \mod 4 \]
and hence
\[16 \nmid  2p^{\ceil{\log_p(2q_k^2)}} - 4q_k^2. \]
\end{proof}

We are now ready to prove a lower bound on the kissing numbers of the graphs $X^{p,q_k}$.

\begin{proof}[Proof of Theorem \ref{thm_kiss_ramanujan} (b)]
Let $p$, $k$ and $q_k$ be as in \lemref{lem:sixteen}. The bounds $p^k < q_k < 2 p^k$ yield
\[ p^{\ceil{\log_p(q_k^2)}} - q_k^2 = p^{2k+1} - q_k^2 > q_k^2/4 \] 
(where we use the fact that $p\geq 5$) and
\[ 2p^{\ceil{\log_p(2q_k^2)}} - 4q_k^2 \geq 2p^{2k+1} - 4q_k^2 > 3 q_k^2/2. \]

By combining  formula \ref{eq_loops_id},  \lemref{lem:sixteen}, and the lower bound in \lemref{lem:BatemanSiegel}, we get that for every $\epsilon>0$ there exists a constant $D_\eps>0$ such that the number of shortest closed geodesics through the identity in $X^{p,q_k}$ is at least 
\[ D_\eps\, q_k^{1-\eps} \geq D_\eps \card{X^{p,q_k}}^{(1-\eps)/3}.\]
If we multiply this by the number of vertices and divide by the girth to compensate for overcounting, we obtain
\[ \kiss(X^{p,q_k}) \geq D_\eps' \frac{\card{X^{p,q_k}}^{(4-\eps)/3}}{\log_p\left(\card{X^{p,q_k}}\right)}\]
which is bounded below by $\card{X^{p,q_k}}^{4/3 - \eps}$ if $k$ is large enough.
\end{proof}

\section{Kissing numbers of small cubic graphs} \label{sec:data}

For the purpose of this section, a \emph{cubic graph} is a connected $3$-regular graph of girth at least $3$. Using the program \texttt{geng} from the computer package \texttt{nauty} developed by Brendan McKay \cite{McKay} (available in \texttt{SageMath} \cite{sagemath}), we generated all 125,816,453 cubic graphs on at most $24$ vertices and computed their kissing number, girth, diameter, and size of automorphism group. This data is compiled in Table \ref{table} and \figref{fig:plots} (found in Appendix \ref{sec:appendix}). Note that the diameter of a graph is the maximal distance between two vertices, which is a priori different from its diameter as a metric space.

\begin{table}[htb]
\begin{tabular}{|l|l|l|l|l|l|l|}
\hline
Vertices & Kissing & Girth & $\#\Aut$ & Diameter  & Relations & Moore \\
\hline
\hline
4 & 8 & 3 & 24 & 1  & $K=G=A=D$ & yes\\
\hline
6 & 18 & 4 & 72 & 2 ($\times$2)  & $K=G=A\subset D$ & yes \\
\hline
8 & 12  & 4 ($\times$2) & 48 & 2 ($\times$2)  & $K=A \subset G$,  $G\cap D$ & no\\
\hline
10 & 24 & 5 & 120 & 2 & $K=G=A=D$  & yes \\
\hline
12 & 20 &  5 ($\times$2) & 64 & 3 ($\times$34) & $K = A$,  $G \subset D$  & no\\
\hline
14 & 56  & 6 & 336 & 3 ($\times$34) & $K=G=A \subset D$ & yes \\
\hline
16 & 48 & 6 & 384  & 3 ($\times$14) & $K=G$ & no \\
\hline
18 & 42 &  6 ($\times$5) & 384 ($\times$2) & 3  & $K \subset G$ & no \\
\hline 
20 & 40 ($\times$3) & 6 ($\times$32) & 768 ($\times$2) & 3  & $K \subset G$ & no \\
\hline
22 & 44 & 6 ($\times$385) & 3072 & 4 ($\times$185836) & $K \subset G$, $G\cap D$  & no \\ 
\hline
24 & 64 & 7 & 3072 & 4 ($\times$341797) & $K = G \subset D$ & no \\ 
\hline
\end{tabular}

\vspace{0.2cm}

\caption{Largest kissing number, girth, and order of automorphism group, and smallest diameter, among all cubic graphs on a given number of vertices. The multiplicity of values that are achieved more than once is written in parentheses. The sets of record holders are denoted by $K$, $G$, $A$, and $D$ respectively and the inclusions and nonempty intersections among them are displayed in the column `Relations'. The column `Moore'  indicates whether there is a cubic Moore graph with the given number of vertices.}
\label{table}
\end{table}

From Table \ref{table}, we observe that among all cubic graphs on $n$ vertices, there is often a unique one with largest kissing number (this fails for 20 vertices) and the graphs with largest kissing number often have the largest girth as well (this fails for $12$ vertices). This happens even if there is no Moore graph on $n$ vertices. For instance, the McGee graph of girth 7---the smallest cubic cage which is not a Moore graph---uniquely maximizes kissing number among all cubic graphs on 24 vertices (or less). Recall that a cage is a regular graph whose number of vertices is smallest possible given its degree and girth.

Figure \ref{fig:plots} plots the different invariants against each other in pairs. In order to compare graphs of various sizes in a meaningful way, we define relative versions of each invariant taking values between $0$ and $1$. These measure the extent to which certain bounds are saturated. Specifically, the \emph{relative kissing number} of a cubic graph $G$ of girth $g$ on $n$ vertices is
\begin{equation*}
\kiss(G) \cdot \frac{2\log_{2}\left(\frac{n + 2}{3}\right)+1}{n(n+2)} \quad\text{if }g\text{ is odd}
\end{equation*}
or
\begin{equation*}
\kiss(G) \cdot \frac{4\log_{2}\left(\frac{n + 2}{2}\right)}{3n(n+2)} \quad\text{if }g\text{ is even}
\end{equation*}
based on \corref{cor:main}. Its \emph{relative girth} is $(3\cdot 2^{(g-1)/2}-2)/n$ if $g$ is odd and $2(2^{g/2}-1)/n$ if $g$ is even, which measures the saturation of the Moore bound. Similarly, the \emph{relative inverse diameter} of a cubic graph of diameter $D$ is $n/(3\cdot 2^D-2)$. This is in reference to the elementary bound 
\begin{equation} \label{eq:diam}
n \leq 1 + d\sum_{j=0}^{D-1} (d-1)^j
\end{equation}
coming from the fact that a connected $d$-regular graph of diameter $D$ is contained in the ball of radius $D$ about any vertex. Note that Moore graphs of odd girth saturate this bound while the Moore graphs of even girth do not. Incidentally, there are non-Moore cubic graphs of smallest diameter on $6$ and $14$ vertices in addition to the Moore graphs of girth $4$ and $6$. Finally, the \emph{relative size of automorphism group} of a cubic graph is the order of its automorphism group divided by Wormald's upper bound of $3(n/2)2^{n/2}$ \cite{Wormald}. Note that Wormald's bound is only attained for graphs on $4$ or $6$ vertices. For cubic graphs on at least $16$ vertices, a sharp improvement was obtained by van Opstall and Veliche \cite{Veliche}. However, their inequality is more complicated (it depends on the binary expansion of $n/2 + 1$) and since we are dealing with small graphs anyway, we use the weaker bound. 

Here are some observations based on the plots produced. \corref{cor:main} states that if the relative girth or the relative kissing number of a graph is equal to $1$, then so is the other. Its proof also shows that if the kissing bound is nearly saturated, then so is the Moore bound. Surprisingly, \figref{fig:0} suggests a relationship in the other direction as well, that is, a relatively large girth seems to imply a relatively large kissing number. A similar but weaker correlation can be observed between the relative size of automorphism group and the relative kissing number in \figref{fig:1}. \figref{fig:2} shows that a graph with relatively large kissing number is forced to have relatively small diameter, but that the converse is false. From Figures \ref{fig:3}, \ref{fig:4} and \ref{fig:5}, it appears that the correlation between the other pairs of variables is somewhat weaker. One might have expected the girth and inverse diameter to display a stronger correlation, given that the proofs of the corresponding bounds are almost identical.

\section{Questions} \label{sec:questions}

We conclude with a list of questions. Define the \emph{depth} of a graph $G$ to be the largest integer $\delta$ such that every geodesic walk of length $\delta$ in $G$ is contained in a shortest closed geodesic in $G$. \lemref{lem:fellow_travel} implies that the depth of a connected graph which is not a cycle is at most $\lfloor g/2 \rfloor + 1$ where $g$ is the girth. Furthermore, the proof of \thmref{thm:main} shows that graphs of large depth have large kissing number. More precisely, any graph $G$ of girth $g$ on $n$ vertices with minimal degree $d$ and depth $\delta \geq 1$ satisfies
\begin{equation} \label{eq:depth}
\kiss(G) \geq \frac{nd(d-1)^{\delta-1}}{g}.
\end{equation}
 The proof of \thmref{thm:main} also shows that Moore graphs are precisely the connected regular graphs of depth $\lfloor g/2 \rfloor + 1$. Recall that apart from cycles, all Moore graphs have girth at most 12 and hence depth at most $7$. As another example, $s$-arc-transitive graphs have depth at least $s$ but Weiss \cite{Weiss2} showed that the largest $s$ can be is $7$. This raises the following questions.

\begin{que} \label{LargeDepth}
Do there exist $d$-regular graphs of arbitrarily large depth for any $d \geq 3$?
\end{que}

\begin{que}
Do there exist graphs of minimal degree at least $3$ and arbitrarily large depth?
\end{que}

A similar question without the regularity or the minimal degree requirement was posed by Neumaier, who asked whether all connected graphs of sufficiently large depth were cycles or uniform subdivisions of other graphs. This was answered in the negative by Homobono and Peyrat \cite{HP}. Their counterexamples have arbitrarily large depth, but still many vertices of degree $2$ (they are subdivisions of other graphs, just not uniform). 

If Question \ref{LargeDepth} has an affirmative answer, then one can ask the following stronger version.

\begin{que}
Does there exist a sequence of $d$-regular graphs of unbounded girth whose depth is a positive proportion of their girth, for any $d \geq 3$?
\end{que}

By inequality \eqref{eq:depth}, such graphs would have super-linear kissing numbers if their girth was logarithmically large. While some of the Ramanujan graphs of Lubotzky--Phillips--Sarnak have super-linear kissing numbers, it is not at all clear if their depth is proportional to their girth. All that is really needed to get large kissing number is that sufficiently many long walks can be extended to shortest closed geodesics, rather than all.

The discrepancy between \corref{cor:main} and  \thmref{thm_kiss_ramanujan_intro} also begs the question of whether one can do better.

\begin{que}
Does there exist, for any $d\geq 3$, a sequence of $d$-regular graphs with kissing numbers growing like the number of vertices to some power $a>4/3$? 
\end{que}

Another interesting question is whether super-linearity can be achieved in \emph{every} degree.

\begin{que} \label{EveryDegree}
Does there exist, for every $d\geq 3$, a sequence of $d$-regular graphs with kissing numbers growing like the number of vertices to some power $a>1$? 
\end{que}

Such graphs necessarily have logarithmically large girth by \thmref{thm:main}. A theorem of Erd\"os and Sachs \cite{ErdosSachs} says that logarithmic girth can be achieved in every degree, and there even exist explicit constructions in all degrees except $7$ \cite{Dahan}. Besides the graphs $X^{p,q}$, the following families would be a good place to start for answering Question \ref{EveryDegree}.

\begin{que}
Does a subsequence of the (cubic) sextet graphs of Biggs and Hoare \cite{BiggsHoare} have super-linear kissing numbers? 
\end{que}

\begin{que}
Does a subsequence of the $(m+1)$-regular graphs of Morgenstern \cite{Morgenstern} have super-linear kissing numbers for every prime power $m$? 
\end{que}

\begin{que}
Does a subsequence of the $(m+1)$-regular graphs of Dahan \cite{Dahan} have super-linear kissing numbers for every $m\geq 10$ which is not a prime power? 
\end{que}

The last question is raised by the numerical data of \secref{sec:data}.

\begin{que}
If a $d$-regular graph maximizes the kissing number among all $d$-regular graphs on at most the same number of vertices, then is it a cage?
\end{que}

%%%%%%%%%%%%%%%%%%%%%%%%%%%%%%%%%%%%%%%%%%%%%%%%%
%		B I B L I O G R A P H Y
%%%%%%%%%%%%%%%%%%%%%%%%%%%%%%%%%%%%%%%%%%%%%%%%%
\bibliography{biblio}

\begin{thebibliography}{BFdlV82}

\bibitem[AK15]{Azakla}
J.~Azarija and S.~Klav\v{z}ar.
\newblock Moore graphs and cycles are extremal graphs for convex cycles.
\newblock {\em J. Graph Theory}, 80(1):34--42, 2015.

\bibitem[Alo86]{Alon}
N.~Alon.
\newblock Eigenvalues and expanders.
\newblock volume~6, pages 83--96. 1986.
\newblock Theory of computing (Singer Island, Fla., 1984).

\bibitem[Bat51]{Bateman}
P.T. Bateman.
\newblock On the representations of a number as the sum of three squares.
\newblock {\em Trans. Amer. Math. Soc.}, 71:70--101, 1951.

\bibitem[BB90]{BiggsBoshier}
N.L. Biggs and A.G. Boshier.
\newblock Note on the girth of {R}amanujan graphs.
\newblock {\em J. Combin. Theory Ser. B}, 49(2):190--194, 1990.

\bibitem[BFdlV82]{BF}
B.~Bollob\'{a}s and W.~Fernandez de~la Vega.
\newblock The diameter of random regular graphs.
\newblock {\em Combinatorica}, 2(2):125--134, 1982.

\bibitem[BH83]{BiggsHoare}
N.L. Biggs and M.J. Hoare.
\newblock The sextet construction for cubic graphs.
\newblock {\em Combinatorica}, 3(2):153--165, 1983.

\bibitem[Big93]{BiggsBook}
N.~Biggs.
\newblock {\em Algebraic graph theory}.
\newblock Cambridge Mathematical Library. Cambridge University Press,
  Cambridge, second edition, 1993.

\bibitem[Big98]{BiggsCubic}
N.~Biggs.
\newblock Constructions for cubic graphs with large girth.
\newblock {\em Electron. J. Combin.}, 5:Article 1, 25, 1998.

\bibitem[CS99]{ConwaySloane}
J.H. Conway and N.J.A. Sloane.
\newblock {\em Sphere packings, lattices and groups}, volume 290 of {\em
  Grundlehren der Mathematischen Wissenschaften [Fundamental Principles of
  Mathematical Sciences]}.
\newblock Springer-Verlag, New York, third edition, 1999.
\newblock With additional contributions by E. Bannai, R.E. Borcherds, J. Leech,
  S.P. Norton, A.M. Odlyzko, R.A. Parker, L. Queen and B.B. Venkov.

\bibitem[Dah14]{Dahan}
X.~Dahan.
\newblock Regular graphs of large girth and arbitrary degree.
\newblock {\em Combinatorica}, 34(4):407--426, 2014.

\bibitem[DSV03]{DSV}
G.~Davidoff, P.~Sarnak, and A.~Valette.
\newblock {\em Elementary number theory, group theory, and {R}amanujan graphs},
  volume~55 of {\em London Mathematical Society Student Texts}.
\newblock Cambridge University Press, Cambridge, 2003.

\bibitem[ES63]{ErdosSachs}
P.~Erd\H{o}s and H.~Sachs.
\newblock Regul\"{a}re {G}raphen gegebener {T}aillenweite mit minimaler
  {K}notenzahl.
\newblock {\em Wiss. Z. Martin-Luther-Univ. Halle-Wittenberg Math.-Natur.
  Reihe}, 12:251--257, 1963.

\bibitem[FBP19]{FBP}
M.~Fortier~Bourque and B.~Petri.
\newblock Kissing numbers of closed hyperbolic manifolds.
\newblock Preprint, \href{https://arxiv.org/abs/1905.11083}{\tt
  arXiv:1905.11083}, 2019.

\bibitem[FP15]{FanPar}
F.~Fanoni and H.~Parlier.
\newblock Systoles and kissing numbers of finite area hyperbolic surfaces.
\newblock {\em Algebr. Geom. Topol.}, 15(6):3409--3433, 2015.

\bibitem[Gro85]{Grosswald}
E.~Grosswald.
\newblock {\em Representations of integers as sums of squares}.
\newblock Springer-Verlag, New York, 1985.

\bibitem[HP89]{HP}
N.~Homobono and C.~Peyrat.
\newblock Graphs such that every two edges are contained in a shortest cycle.
\newblock {\em Discrete Math.}, 76(1):37--44, 1989.

\bibitem[Imr84]{Imrich}
W.~Imrich.
\newblock Explicit construction of regular graphs without small cycles.
\newblock {\em Combinatorica}, 4(1):53--59, 1984.

\bibitem[LPS88]{LPS}
A.~Lubotzky, R.~Phillips, and P.~Sarnak.
\newblock Ramanujan graphs.
\newblock {\em Combinatorica}, 8(3):261--277, 1988.

\bibitem[Lub10]{Lubotzky}
A.~Lubotzky.
\newblock {\em Discrete groups, expanding graphs and invariant measures}.
\newblock Modern Birkh\"{a}user Classics. Birkh\"{a}user Verlag, Basel, 2010.
\newblock With an appendix by Jonathan D. Rogawski, Reprint of the 1994
  edition.

\bibitem[Mar82]{Margulis}
G.A. Margulis.
\newblock Explicit constructions of graphs without short cycles and low density
  codes.
\newblock {\em Combinatorica}, 2(1):71--78, 1982.

\bibitem[Mar88]{Margulis2}
G.A. Margulis.
\newblock Explicit group-theoretic constructions of combinatorial schemes and
  their applications in the construction of expanders and concentrators.
\newblock {\em Problemy Peredachi Informatsii}, 24(1):51--60, 1988.

\bibitem[McK81]{McKay}
B.D. McKay.
\newblock Practical graph isomorphism.
\newblock {\em Congr. Numer.}, 30:45--87, 1981.

\bibitem[Mor94]{Morgenstern}
M.~Morgenstern.
\newblock Existence and explicit constructions of {$q+1$} regular {R}amanujan
  graphs for every prime power {$q$}.
\newblock {\em J. Combin. Theory Ser. B}, 62(1):44--62, 1994.

\bibitem[Par13]{Par}
H.~Parlier.
\newblock Kissing numbers for surfaces.
\newblock {\em J. Topol.}, 6(3):777--791, 2013.

\bibitem[Sar90]{Sarnak}
P.~Sarnak.
\newblock {\em Some applications of modular forms}, volume~99 of {\em Cambridge
  Tracts in Mathematics}.
\newblock Cambridge University Press, Cambridge, 1990.

\bibitem[Sar19]{Sardari}
N.T. Sardari.
\newblock Diameter of {R}amanujan graphs and random {C}ayley graphs.
\newblock {\em Combinatorica}, 39(2):427--446, 2019.

\bibitem[Sch94]{SchmutzKiss}
P.~Schmutz.
\newblock Systoles on {R}iemann surfaces.
\newblock {\em Manuscripta Math.}, 85(3-4):429--447, 1994.

\bibitem[Sie15]{Siegel}
C.L. Siegel.
\newblock {\em Gesammelte {A}bhandlungen. {I}}.
\newblock Springer Collected Works in Mathematics. Springer, Heidelberg, 2015.
\newblock Edited by Komaravolu Chandrasekharan and Hans Maa\ss , Reprint of the
  1966 edition [ MR0197270].

\bibitem[SS97]{SchmutzSuper}
P.~Schmutz~Schaller.
\newblock Extremal {R}iemann surfaces with a large number of systoles.
\newblock In {\em Extremal {R}iemann surfaces ({S}an {F}rancisco, {CA}, 1995)},
  volume 201 of {\em Contemp. Math.}, pages 9--19. Amer. Math. Soc.,
  Providence, RI, 1997.

\bibitem[{The}19]{sagemath}
{The Sage Developers}.
\newblock {\em {S}ageMath, the {S}age {M}athematics {S}oftware {S}ystem
  ({V}ersion 8.7)}, 2019.
\newblock {\tt https://www.sagemath.org}.

\bibitem[TK92]{TeoKoh}
C.P. Teo and K.M. Koh.
\newblock The number of shortest cycles and the chromatic uniqueness of a
  graph.
\newblock {\em J. Graph Theory}, 16(1):7--15, 1992.

\bibitem[Vl{\u a}18]{Vladut}
S.~Vl{\u a}du{\c t}.
\newblock Lattices with exponentially large kissing numbers.
\newblock Preprint, \href{https://arxiv.org/abs/1802.00886}{\tt
  arXiv:1802.00886}, 2018.

\bibitem[vOV10]{Veliche}
M.A. van Opstall and R.~Veliche.
\newblock Cubic graphs with most automorphisms.
\newblock {\em J. Graph Theory}, 64(2):99--115, 2010.

\bibitem[Wal36]{Walfisz}
A.~Walfisz.
\newblock Zur additiven {Z}ahlentheorie. {II}.
\newblock {\em Math. Z.}, 40(1):592--607, 1936.

\bibitem[Wei74]{Weiss2}
R.~M. Weiss.
\newblock \"{U}ber {$s$}-regul\"{a}re {G}raphen.
\newblock {\em J. Combinatorial Theory Ser. B}, 16:229--233, 1974.

\bibitem[Wei84]{Weiss}
A.~Weiss.
\newblock Girths of bipartite sextet graphs.
\newblock {\em Combinatorica}, 4(2-3):241--245, 1984.

\bibitem[Wor79]{Wormald}
N.~Wormald.
\newblock On the number of automorphisms of a regular graph.
\newblock {\em Proc. Amer. Math. Soc.}, 76(2):345--348, 1979.

\end{thebibliography}
\bibliographystyle{alpha}

\appendix
\section{Figure} \label{sec:appendix}

\begin{figure}[H]

\begin{subfigure}{.5\textwidth}
\centering
\includegraphics[scale=.9]{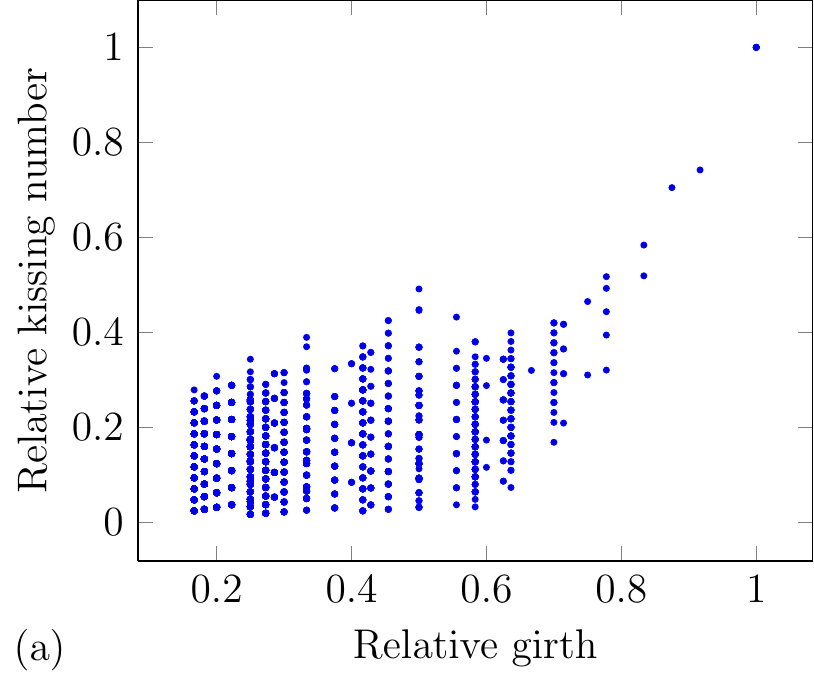}
\phantomsubcaption
\label{fig:0}
\end{subfigure}%
\begin{subfigure}{.5\textwidth}
\centering
\includegraphics[scale=.9]{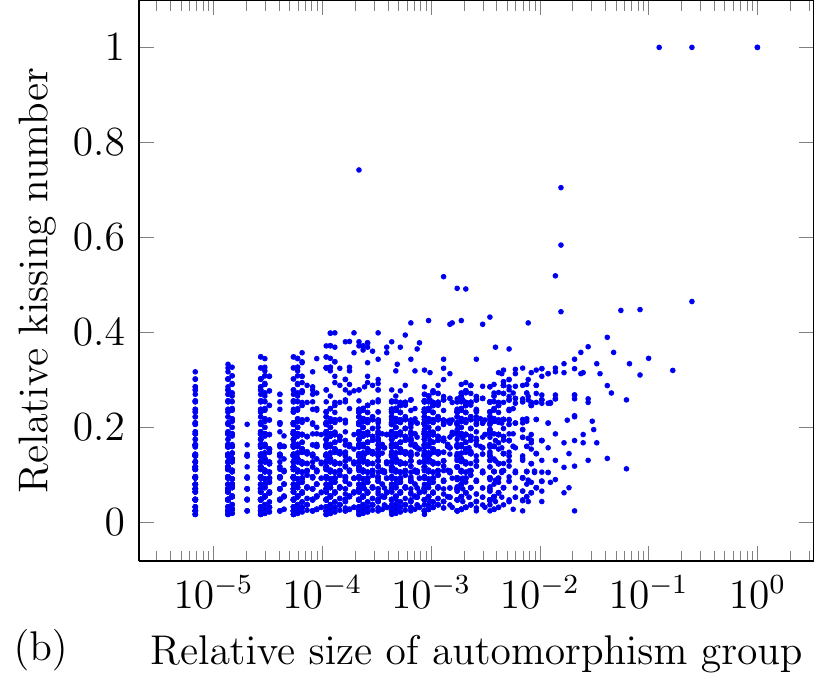}
\phantomsubcaption
\label{fig:1}
\end{subfigure}\\[3ex]
\begin{subfigure}{.5\textwidth}
\centering
\includegraphics[scale=.9]{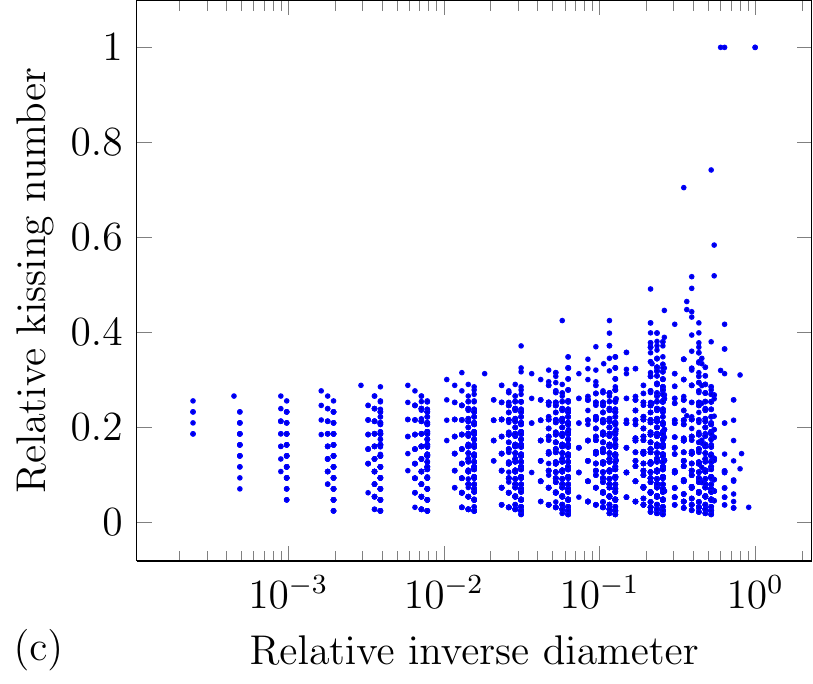}
\phantomsubcaption
\label{fig:2}
\end{subfigure}%
\begin{subfigure}{.5\textwidth}
\centering
\includegraphics[scale=.9]{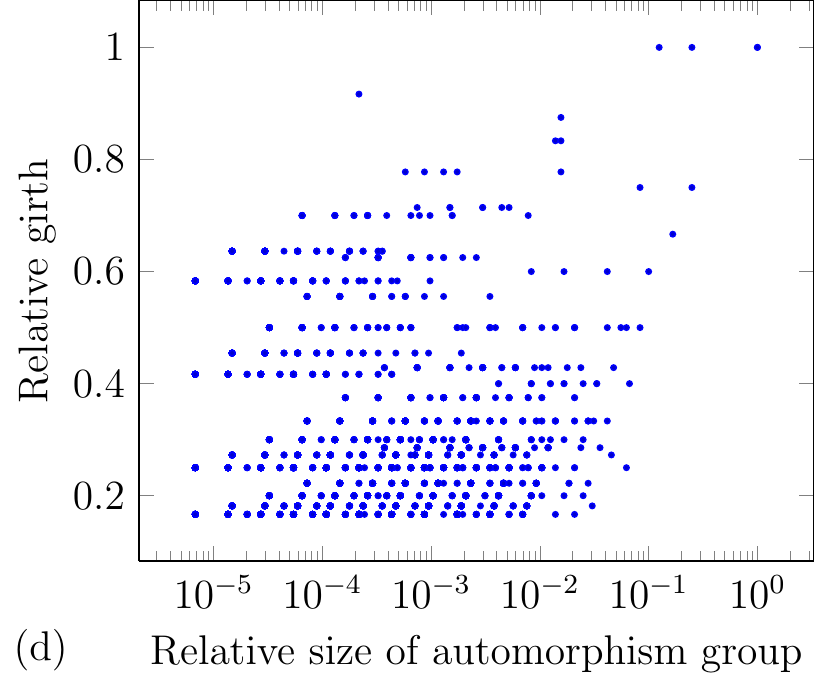}
\phantomsubcaption
\label{fig:3}
\end{subfigure}\\[3ex]
\begin{subfigure}{.5\textwidth}
\centering
\includegraphics[scale=.9]{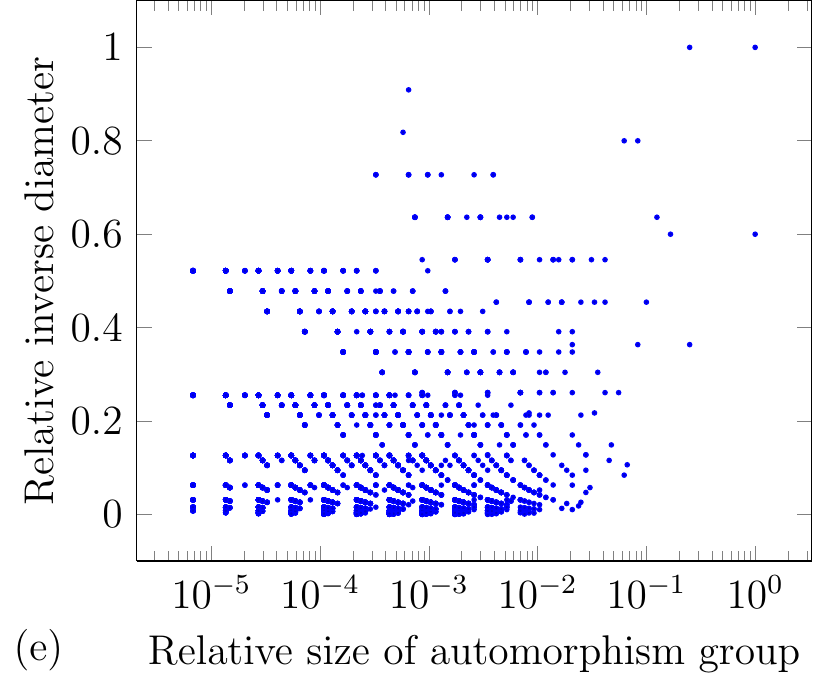}
\phantomsubcaption
\label{fig:4}
\end{subfigure}%
\begin{subfigure}{.5\textwidth}
\centering
\includegraphics[scale=.9]{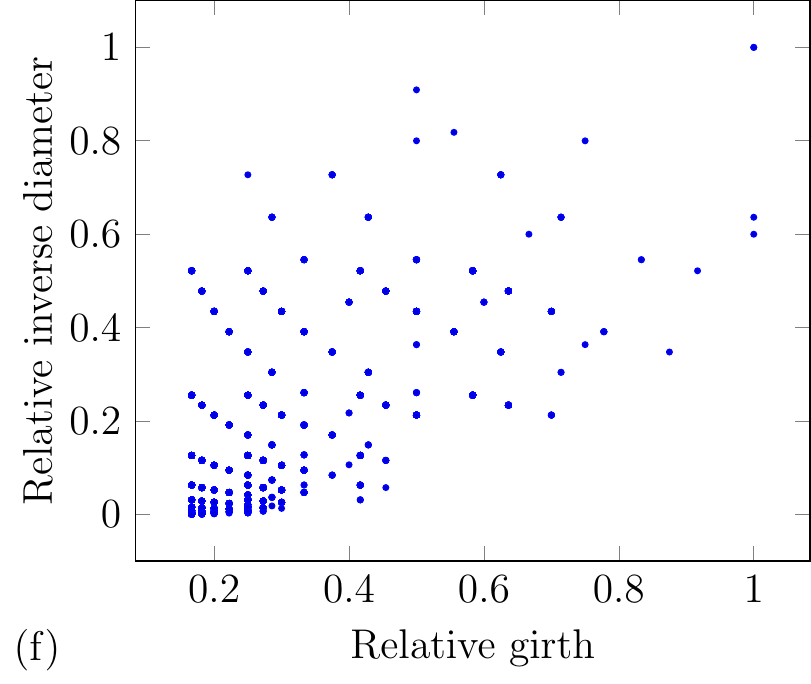}
\phantomsubcaption
\label{fig:5}
\end{subfigure}

\caption{Scatter plots comparing the relative invariants of all cubic graphs on at most 24 vertices}

\label{fig:plots}

\end{figure}

\end{document}